\documentclass[1p]{elsarticle}

\usepackage{lineno}
\modulolinenumbers[5]

\usepackage{mathrsfs}
\usepackage{amsfonts}
\usepackage[notref,notcite]{showkeys}
\usepackage{enumerate}
\usepackage{latexsym}
\usepackage[all,cmtip]{xy}
\usepackage{verbatim}
\usepackage{xcolor}

\usepackage{ifpdf}
\ifpdf
 \usepackage[colorlinks,final,backref=page,hyperindex]{hyperref}
\else
  \usepackage[colorlinks,final,backref=page,hyperindex,hypertex]{hyperref}
\fi

\usepackage{amsmath,amssymb,xspace,amsthm}

\numberwithin{equation}{section}

\def\fa{\mathfrak{a}}

\def\fh{\mathfrak{h}}

\def\fm{\mathfrak{m}}

\def\cK{\mathcal{K}}

\def\cU{\mathcal{U}}
\def\cI{\mathcal{I}}
\def\cT{\mathcal{T}}

\def\bC{\mathbb{C}}
\def\bN{\mathbb{N}}
\def\bZ{\mathbb{Z}}

\newcommand\ol{\overline}
\newcommand\wt{\widetilde}

\newtheorem{theo}{{Theorem}}[section]
\newtheorem{lemm}[theo]{Lemma}

\newtheorem{defi}[theo]{Definition}

\newtheorem{prop}[theo]{Proposition}

\allowdisplaybreaks

\begin{document}

\begin{frontmatter}



\title{Simple Harish-Chandra modules over the superconformal current algebra}
\author[1]{Yan He}
\ead{heyan913533012@163.com}

\author[2]{Dong Liu}
\ead{liudong@zjhu.edu.cn}

\author[3]{Yan Wang}
\ead{wangyan09@tju.edu.cn,  corresponding author}
\address[1]{Department of Mathematics, Changshu Institute of Technology, Jiangsu P. R. China}
\address[2]{Department of Mathematics, Huzhou University, Zhejiang P. R. China}
\address[3]{School of Mathematics, Tianjin University, Tianjin, P. R. China}

\begin{abstract}
In this paper, we classify the simple Harish-Chandra modules over the superconformal current algebra $\widehat{\frak g}$, which is the semi-direct sum of the $N=1$ superconformal algebra with the affine Lie superalgebra $\dot{\frak g} \otimes \mathcal{A}\oplus \mathbb CC_1$, where $\dot{\frak g}$ is a finite-dimensional simple Lie algebra, and $\mathcal{A}$ is the tensor product of the Laurent polynomial algebra and the Grassmann algebra. As an application, we can directly get the classification of the simple Harish-Chandra modules over the $N=1$ Heisenberg-Virasoro algebra.
\end{abstract}

\date{ }
\end{frontmatter}

\section{Introduction}

It is well known that the Virasoro algebra is the simplest case of the superconformal algebras (corresponding to $N=0$)(\cite{KV}). The Virasoro algebra has been widely used in many fields of mathematics and physics (see \cite{IK, KA} and the references therein). There is more than one super extension of the Virasoro algebra since early 70's.
The simplest are the $N=1$ superconformal algebras including the $N=1$ Neveu-Schwarz and Ramond algebras. These infinite-dimensional Lie superalgebras are also called the super-Virasoro algebras, since they can be viewed as supersymmetry extensions of the Virasoro algebra. The complete classification of simple Harish-Chandra modules over the $N=1$ superconformal algebras was given in \cite{S}. Recently, by using a simpler approach, the classification of simple Harish-Chandra modules over the $N=1$ Neveu-Schwarz and Ramond algebras was given in \cite{CLL} and \cite{CL}, respectively.

As a supersymmetric extension of affine-Virasoro algebra, the superconformal current algebra $\widehat{\frak g}$ introduced in \cite{KT} is the semi-direct sum of the $N=1$ superconformal algebra with the affine Lie superalgebra $\dot{\frak g}\otimes \mathcal{A}\oplus\mathbb CC_1$, where $\dot{\frak g}$ is a finite-dimensional simple Lie algebra, and $\mathcal{A}$ is the tensor product of the Laurent polynomial algebra and the Grassmann algebra. This supersymmetric extension is implicit in recent work on superstring also admits a local field interpretation (\cite{W}). Here, we will consider the Ramond case. Recently many researches of the superconformal current algebra have appeared in mathematics and physics, see the references for details (\cite{CHL, G, GSWX, LPXZ, Ma, MOR}).

Based on the classification of simple jet modules introduced by Y. Billig in \cite{B} (also see \cite{E}), a complete classification of simple Harish-Chandra modules over Lie algebra
of vector fields on a torus was given in \cite{BF} by the $\mathcal{A}$-cover theory. Recently, with the $\mathcal{A}$-cover theory, all simple Harish-Chandra modules were classified over the Witt superalgebra (\cite{XL2}) (also see \cite{BFI}), the $N=1$ superconformal algebra (\cite{CLL, CL}) and the map (super)algebra related to the Virasoro
algebra (\cite{CLW}). In this paper, we use this method to classify simple Harish-Chandra modules over the superconformal current algebra $\widehat{\frak g}$ and get the following main result.

\noindent{\bf Main theorem} (Theorem \ref{main}) {\it
Let $M$ be a simple Harish-Chandra module over the superconformal current algebra $\widehat{\frak g}$.
Then $M$ is a highest weight module, a lowest weight module, or a simple quotient of a tensor module $\Gamma(\lambda, V, b)$ or $\Pi(\Gamma(\lambda, V, b))$
for some finite-dimensional simple $\dot{\frak g}$-module $V$ and $\lambda, b\in\bC$.
}

As an application, we can get the classification of simple Harish-Chandra modules over the $N = 1$ Heisenberg-Virasoro algebra (Theorem \ref{main1}), which is the universal central extension of the Green-Schwarz-Witten (GSW) superalgebra (\cite{AJR, PB}). The GSW superalgebra was introduced in the study of superstring theory with the name ``the complete spectrum-generating superalgebra'' (\cite{GSW}).

The paper is organized as follows. In Section 2, we give some definitions and preliminaries. In Section 3, we introduce the definitions and basic conclusions of the $\mathcal{A}$-cover. In Section 4, we study the simple cuspidal modules of $\widehat{\frak g}$. In Section 5, we classify the simple Harish-Chandra modules over $\widehat{\frak g}$ and present the results for the $N = 1$ Heisenberg-Virasoro algebra.

\section{Preliminaries}
Throughout this paper, we denote by $\mathbb{Z}, \mathbb{Z}_+, \mathbb{N}$ and $\bC$ the sets of all integers, nonnegative integers, positive integers and complex numbers, respectively. All algebras, modules and vector spaces are assumed to be over $\bC$. A super vector space $V=V_{\bar{0}}\oplus V_{\bar{1}}$ is a vector space endowed with a $\mathbb{Z}_2$-gradation. The parity of a homogeneous element $v\in V_{\bar{i}}$ is denoted by $|v|=\bar{i}\in \mathbb{Z}_2$. When we write $|v|$
for an element $v\in V$, we will always assume that $v$ is a homogeneous element.
\begin{defi}
The $N=1$ Ramond algebra $\frak s$ is a Lie superalgebra with basis $\{L_i, G_i, C\mid i\in \mathbb{Z}\}$ and brackets
\begin{align*}
&[L_i,L_j]=(j-i)L_{i+j}+\delta_{i+j,0}\frac{1}{12}(i^3-i)C,\\
&[L_i,G_j]=(j-\frac{i}{2})G_{i+j},\\
&[G_i,G_j]=-2L_{i+j}+\delta_{i+j,0}\frac{1}{12}(4i^2-1)C
\end{align*} for any $i, j\in\mathbb Z$.
\end{defi}
The even part of $\frak s$ is spanned by $\{L_i, C\mid i\in \mathbb{Z}\}$, and is isomorphic to the Virasoro algebra, which is the universal central extension of the Witt algebra
$\mathcal{W}$. The odd part of $\frak s$ is spanned by $\{G_i\mid i\in \mathbb{Z}\}$. Let ${\overline {\frak s}}$ be the quotient algebra $\frak s/\bC C$.

Let $\mathcal{A}=\bC[t^{\pm 1}]\otimes\Lambda(1)$ be the tensor product of the Laurent polynomial algebra in even variable $t$ and the Grassmann algebra in odd variable $\xi$.
$\mathcal{A}$ is an ${\overline {\frak s}}$-module with
\begin{align*}
&L_i\cdot f=t^{i+1}\frac{\partial}{\partial t}(f)+\frac{i}{2}t^{i}\xi\frac{\partial}{\partial\xi}(f),\\
&G_i\cdot f=t^{i+1}\xi\frac{\partial}{\partial t}(f)-t^{i}\frac{\partial}{\partial\xi}(f),
\end{align*}
where $i\in \bZ, f\in \mathcal{A}$. We have $[a, f]=a\cdot f$ for $a\in {\overline {\frak s}}$ and $f\in \mathcal{A}$.
On the other hand, the paper \cite{CLL} shows that $\overline {\frak s}$ has an $\mathcal{A}$-module structure
\begin{align*}
t^{i}L_n=L_{n+i},\ t^{i}G_n=G_{n+i},\ \xi L_n=\frac{1}{2}G_n,\ \xi G_n=0,\ \forall i, n \in \bZ.
\end{align*}
For a finite-dimensional simple Lie algebra $\dot{\frak g}$, we have $\dot{\frak g}=[\dot{\frak g}, \dot{\frak g}]$ and $\dot{\frak g}$ has a non-degenerated invariant bilinear form
$(\cdot~|~\cdot)$. In this case $\dot{\frak g}\otimes \mathcal{A}$ is a Lie superalgebra, which is called the super loop algebra (see \cite{KT}), with $$[x\otimes f, y\otimes g]=[x, y]\otimes fg$$ for all $x, y\in \dot{\frak g}, f,g\in \mathcal{A}$. Obviously, $\dot{\frak g}\otimes \mathcal{A}$ has a natural $\mathcal{A}$-module structure
\begin{align*}
&t^{i}(x\otimes t^n)=x\otimes t^{n+i},\ t^{i}(x\otimes t^n\xi)=x\otimes t^{n+i}\xi,\nonumber\\
&\xi(x\otimes t^n)=x\otimes t^{n}\xi,\ \ \xi(x\otimes t^n\xi)=0
\end{align*}
for $i, n\in\bZ$ and $x\in\dot{\frak g}$. Meanwhile, $\dot{\frak g}\otimes \mathcal{A}$ becomes an $\overline {\frak s}$-module by the usual actions of $\overline {\frak s}$ on $\mathcal{A}$.
So we can define a Lie superalgebra $\frak g$ associated to $\dot{\frak g}$ as $\frak g=\overline {\frak s}\ltimes(\dot{\frak g}\otimes \mathcal{A})$.
Since $[\frak g, \frak g]=\frak g$,  we can get the universal central extension $\widehat{\frak g}$ of $\frak g$ by direct calculation, which is just the superconformal current algebra in \cite{KT}.

\begin{defi}\cite{KT}
Let $\dot{\frak g}$ be a finite-dimensional simple Lie algebra with a non-degenerated invariant bilinear form
$(\cdot~|~\cdot)$, then the superconformal current algebra $\widehat{\frak g}=\frak s\ltimes (\dot{\frak g}\otimes \mathcal A\oplus\mathbb CC_1)$ with the following brackets:
\begin{align*}
&[L_i,L_j]=(j-i)L_{i+j}+\delta_{i+j,0}\frac{1}{12}(i^3-i)C, \\
&[L_i,G_j]=(j-\frac{i}{2})G_{i+j},\\
&[G_i,G_j]=-2L_{i+j}+\delta_{i+j,0}\frac{1}{12}(4i^2-1)C, \\
&[L_i,x\otimes t^j\xi]=(j+\frac{i}{2})x\otimes t^{i+j}\xi,\\
&[L_i,x\otimes t^j]=jx\otimes t^{i+j},\quad [G_i,x\otimes t^j\xi]=-x\otimes t^{i+j},\\
&[G_i,x\otimes t^j]=jx\otimes t^{i+j}\xi,\\
&[x\otimes t^i\xi, y\otimes t^j]=[x,y]\otimes t^{i+j}\xi,\\
&[x\otimes t^i, y\otimes t^j]=[x,y]\otimes t^{i+j}+i\delta_{i+j,0}(x,y)C_1,\\
&[x\otimes t^i\xi, y\otimes t^j\xi]=\delta_{i+j,0}(x,y)C_1,
\end{align*}
where $i,j\in\bZ$ and $x,y\in\dot{\frak g}$.
\end{defi}

It is clear that $\widehat{\frak g}$ has a $\mathbb{Z}$-grading by the eigenvalues of the adjoint action of $L_0$. Then
$$\widehat{\frak g}=\bigoplus_{n\in\mathbb{Z}}\widehat{\frak g}_n=\widehat{\frak g}_+\oplus \widehat{\frak g}_0\oplus \widehat{\frak g}_-,$$
where
$$\widehat{\frak g}_{\pm}=\bigoplus_{n\in\mathbb{N}}\widehat{\frak g}_{\pm n},\ \ \widehat{\frak g}_0=\dot{\frak g}\oplus \mathbb{C}C\oplus\mathbb{C}C_1\oplus\mathbb{C}L_0\oplus\mathbb{C}G_0.$$
Let $\dot{\fh}$ be the Cartan subalgebra of $\dot{\frak g}$. Then $\widehat{\fh}=\dot{\fh}\oplus \mathbb{C}C\oplus\mathbb{C}C_1\oplus\mathbb{C}L_0$ is the cartan subalgebra of $\widehat{\frak g}$. A highest weight module over $\widehat{\frak g}$ is characterized by its highest weight $\Lambda\in \widehat{\fh}^*$ and highest weight vector $v_0$ such that
$(\widehat{\frak g}_+\oplus\dot{\frak g}_+)v_0=0$ and $hv_0=\Lambda(h)v_0, \forall h\in \widehat{\fh}$.

Set $\wt{\frak g}={\frak g}\ltimes\mathcal{A}$. A $\wt{\frak g}$-module $V$ is called an $\mathcal{A}\frak g$-module if $\mathcal{A}$ acts associatively,
i.e.,
$$t^0 v=v, fgv=f(gv),\forall f,g\in \mathcal{A}, v\in V.$$
Let $\mathfrak{l}$ be any of the Lie superalgerba $\frak g, \widehat{\frak g}$ and $\wt{\frak g}$. An
$\mathfrak{l}$-module $V$ is called a weight module if the action of $L_0$ on $V$ is diagonalizable, i.e., $V=\bigoplus\limits_{\lambda\in \bC} V_{\lambda}$, where
$V_{\lambda}=\{v\in V\mid L_0 v=\lambda v\}$.  The support of a weight module $V$ is defined by $\mbox{supp}(V)=\{\lambda\in \bC\mid V_{\lambda}\ne 0\}$. A weight
$\mathfrak{l}$-module $V$ is called Harish-Chandra if $\dim V_{\lambda}<\infty,\forall \lambda\in \mbox{supp}(V)$, and is called cuspidal or uniformly bounded if there exists
some $N\in \bN$ such that $\dim V_{\lambda}\le N,\forall \lambda\in \mbox{supp}(V)$. Clearly, if $V$ is simple, then $\mbox{supp}(V)\subseteq\lambda+\bZ$ for some $\lambda\in\bC$.

Let $\sigma: L\rightarrow L'$ be any homomorphism of Lie superalgebras or associative superalgebras and $V$ be any $L'$-module. Then $V$ becomes an $L$-module, denoted
by $V^\sigma$, with $x\cdot v=\sigma(x)v$ for $x\in L, v\in V$. Denote $T$ be the automorphism of $L$ defined by $T(x)=(-1)^{|x|}x, \forall x\in L$. For any $L$-module $V$,
$\Pi(V)$ is the module defined by a parity-change of $V$ with $\Pi(V)_{\bar{0}}=V_{\bar{1}}, \Pi(V)_{\bar{1}}=V_{\bar{0}}$ and the actions of $L$ on $\Pi(V)$ is the same as that of $V$.

A module $M$ over an associative superalgebra $B$ is called strictly simple if it is a simple module over the associative algebra $B$ (forgetting the $\bZ_2$-gradation).

We need the following results on tensor modules over tensor superalgebras.
\begin{lemm}\cite[Lemma 2.1, 2.2]{XL2}\label{lm:2.1}
Let $B, B'$ be unital associative superalgebras, and $M, M'$ be $B, B'$ modules, respectively.

1. $M\otimes M'\cong \Pi(M)\otimes \Pi(M'^T)$ as $B\otimes B'$-modules.

2. If $B'$ has a countable basis and $M'$ is strictly simple, then

(1) any $B\otimes B'$-submodule of $M\otimes M'$ is of the form $N\otimes M'$ for some $B$-submodule $N$ of $M$;

(2) any simple quotient of the $B\otimes B'$-module $M\otimes M'$ is isomorphic to some $\overline{M}\otimes M'$ for some simple quotient $\overline{M}$ of $M$;

(3) $M\otimes M'$ is a simple $B\otimes B'$-module if and only if $M$ is a simple $B$-module;

(4) if $V$ is a simple $B\otimes B'$-module containing a strictly simple $B'=\bC\otimes B'$-module $M'$, then $V\cong M\otimes M'$ for some simple $B$-module
$M$.
\end{lemm}

\section{$\mathcal{A}$-cover}
As we all know, the key to the classification of the Harish-Chandra module of $\widehat{\frak g}$ is the classification of its cuspidal module.
Let $V$ be a simple Harish-Chandra module over $\widehat{\frak g}$. By Schur's Lemma, we may assume that the central elements $C$ and $C_1$ act on $V$ by scalars $c$ and $c_1$ respectively.
\begin{lemm}\label{lm:3.1}
Suppose that $V=\bigoplus_{\lambda\in \mathbb{C}}V_\lambda$ is a simple cuspidal module over $\widehat{\frak g}$. Then the actions of central elements $C$ and $C_1$
on $V$ are trival.
\end{lemm}
\begin{proof}
Since $\widehat{\frak g}$ has a Virasoro subalgebra with basis $\{L_i, C\mid i\in \mathbb{Z}\}$, and an untwisted affine Lie subalgebra with basis $\{x\otimes t^i, C_1\mid x\in\dot{\frak g}, i\in \mathbb{Z}\}$. According to Lemma 3.2 of \cite{LPX}, we have $c=c_1=0$.
\end{proof}

From the above lemma, we know that the category of simple cuspidal $\widehat{\frak g}$-modules is naturally equivalent to the category of simple cuspidal $\frak g$-modules.
In order to study cuspidal modules of $\frak g$, we introduce the concept of $\mathcal{A}$-cover. For simplicity, we denote $\dot{\frak g}\otimes \mathcal{A}$ by $I$.

\begin{prop}\label{prop:3.2}
Let $M$ be a $\frak g$-module. Then, the tensor product $I\otimes M$ is an $\mathcal{A}\frak g$-module by the following actions of $\frak g$ and $\mathcal{A}$
\begin{eqnarray}
a(w\otimes v)&=&[a,w]\otimes v+(-1)^{|a||w|}w\otimes a(v),\label{eq3.1}\\
f(w\otimes v)&=&f(w)\otimes v, \label{eq3.2}
\end{eqnarray}
for $a\in\frak g, f\in \mathcal{A}, w\in I$ and $v\in M$. Furthermore, if $M$ is a weight module, then so is $I\otimes M$.
\end{prop}
\begin{proof}
Note that $I$ is the ideal of $\frak g$. So $I\otimes M$ is a $\frak g$-module with the adjoint $\frak g$-action and $\mathcal{A}$-module acting as (\ref{eq3.2}).
Hence $I\otimes M$ is a $\wt{\frak g}
$-module because
\begin{eqnarray*}
&&a(f(w\otimes v))-(-1)^{|a||f|}f(a(w\otimes v))\\
&=&a(f(w)\otimes v)-(-1)^{|a||f|}f([a,w]\otimes v+(-1)^{|a||w|}w\otimes a(v))\\
&=&[a,f(w)]\otimes v+(-1)^{|a||f(w)|}f(w)\otimes a(v)-(-1)^{|a||f|}f([a,w])\otimes v\\
&&-(-1)^{|a|(|f|+|w|)}f(w)\otimes a(v)\\
&=&a(f)(w)\otimes v+(-1)^{|a||f|}f([a,w])\otimes v-(-1)^{|a||f|}f([a,w])\otimes v\\
&=&[a,f](w\otimes v)
\end{eqnarray*}
for homogeneous $a\in\frak g, f\in \mathcal{A}, w\in I$ and $v\in M$. Obviously, the action of $\mathcal{A}$ on $I\otimes M$ is associative. So $I\otimes M$ is an $\mathcal{A}\frak g$-module.

Suppose $M$ is a weight module, then $I\otimes M$ is spanned by $w\otimes v$ with $w$ and $v$ being eigenvectors with respect to the action of $L_0$. By (\ref{eq3.1}), we get that $w\otimes v$ is also an eigenvector for the action of $L_0$. So $I\otimes M$ is a weight module.
\end{proof}

Let $K(M)=\{\sum\limits_iw_i\otimes v_i\in I\otimes M\mid\sum\limits_i(fw_i)v_i=0,\ \forall f\in \mathcal{A}\}$. Then it is easy to see that $K(M)$ is an $\mathcal{A}\frak g$-submodule of $I\otimes M$. Hence we have the $\mathcal{A}\frak g$-module $\widehat{M}=(I\otimes M)/K(M)$. As in \cite{BF}, we call $\widehat{M}$ the $\mathcal{A}$-cover of $M$ if $IM=M$.

Clearly, the linear map
\begin{equation*}\begin{split}
\pi: \ \ \widehat{M}  &\to\ \ \ I M,\\
         w\otimes v+K(M)\ & \mapsto\ \  wv,\quad \forall\ w\in I, v\in M
\end{split}
\end{equation*}
is a $\frak g$-module epimorphism.

Recall that in \cite{BF}, the authors show that every cuspidal $\mathcal{W}$-module is annihilated by the operators $\Omega_{k,p}^{(m)}$ for $m$ large enough.
\begin{lemm}\cite[Corollary 3.7]{BF}\label{lm:3.2}
For every $l\in\bN$ there exists $m\in\bN$ such that for all $k, p\in\bZ$ the differentiators $\Omega_{k, p}^{(m)}=\sum\limits_{i=0}^m(-1)^i\binom{m}{i}L_{k-i}L_{p+i}$ annihilate
every cuspidal $\mathcal{W}$-module with a composition series of length $l$.
\end{lemm}

Let $M$ be a cuspidal $\frak g$-module. Obviously, $M$ is a cuspidal $\mathcal{W}$-module and hence there exists $m\in\bN$ such that $\Omega_{k,p}^{(m)}M=0, \forall k,p\in\bZ$.
Therefore, $[\Omega_{k,p}^{(m)},x\otimes t^j]M=0, \forall k,p,j\in\bZ, x\in\dot{\frak g}$. In Lemma 4.4 of {\cite{CLW}}, the authors show that
$$\sum\limits_{i=0}^{m+2}(-1)^i\binom{m+2}{i}(x\otimes t^{k+j+1-i})L_{p-1+i}=0$$
on $M$. Similarly, from $[\Omega_{k,p}^{(m)},x\otimes t^j\xi]M=0$ and $[L_i,x\otimes t^j\xi]=(j+\frac{i}{2})x\otimes t^{i+j}\xi, \forall i,j,k,p\in\bZ, x\in\dot{\frak g}$, we have
\begin{align*}
0=&[\Omega_{k,p-1}^{(m)},x\otimes t^{j+1}\xi]-2[\Omega_{k,p}^{(m)},x\otimes t^j\xi]+[\Omega_{k,p+1}^{(m)},x\otimes t^{j-1}\xi]-[\Omega_{k+1,p-1}^{(m)},x\otimes t^j\xi]\\
&+2[\Omega_{k+1,p}^{(m)},x\otimes t^{j-1}\xi]-[\Omega_{k+1,p+1}^{(m)},x\otimes t^{j-2}\xi]\\
=&\frac{1}{2}\sum\limits_{i=0}^m(-1)^i\binom{m}{i}\big(x\otimes t^{k+j+1-i}\xi L_{p-1+i}-2x\otimes t^{k+j-i}\xi L_{p+i}+x\otimes t^{k-i-1+j}\xi L_{p+1+i})\\
=&\frac{1}{2}\sum\limits_{i=0}^{m+1}(-1)^{i-1}\binom{m+2}{i+1}x\otimes t^{k+j-i}\xi L_{p+i}+x\otimes t^{k+j+1}\xi L_{p-1}\\
=&\frac{1}{2}\sum\limits_{i=0}^{m+2}(-1)^i\binom{m+2}{i}x\otimes t^{k+j+1-i}\xi L_{p-1+i}
\end{align*}
on $M$. Thus, we have the following lemma.
\begin{lemm}\label{lm:3.3}
Let $M$ be a cuspidal $\frak g$-module. Then there exists $m\in\bN$ such that for all $k, p\in\bZ$ and $x\in\dot{\frak g}$, the operators
$\sum\limits_{i=0}^{m}(-1)^i\binom{m}{i}yL_{p+i}$ annihilate $M$, where $y\in \{x\otimes t^{k-i}, x\otimes t^{k-i}\xi \}.$
\end{lemm}
\begin{prop}\label{prop:3.5}
Let $M$ be a cuspidal module for the Lie superalgebra $\frak g=\overline {\frak s}\ltimes(\dot{\frak g}\otimes \mathcal{A})$.
Then its $\mathcal{A}$-cover $\widehat{M}$ is cuspidal.
\end{prop}
\begin{proof}
Since $\widehat{M}$ is an $\mathcal{A}$-module, it is sufficient to show that one of its nonzero weight spaces is finite-dimensional. Fix a nonzero weight $\alpha+p,\ p\in \mathbb{Z}$ and let us prove that $\widehat{M}_{\alpha+p}=\mbox{span}\{(x\otimes t^{p-k})\otimes M_{\alpha+k}, (x\otimes t^{p-k}\xi) \otimes M_{\alpha+k}\mid k\in\mathbb{Z}, x\in\dot{\frak g}\}$ is finite-dimensional.

We will prove by induction on $|q|$ for $q\in \mathbb{Z}$ that for all $u\in M_{\alpha+q}$,
\begin{align*}
(x\otimes t^{p-q})\otimes u, (x\otimes t^{p-q}\xi)\otimes u
\in &\sum\limits_{|k|\le\frac{m}{2}}\Big((x\otimes t^{p-k})\otimes M_{\alpha+k}+(x\otimes t^{p-k}\xi) \otimes M_{\alpha+k}\Big)\\
&+K(M).
\end{align*}
Assume that this is true for $|q'|<|q|$. Now let us prove that this is true for $q$.
If $|q|\leq \frac{m}{2}$, the claim holds. If $|q|>\frac{m}{2}$, we may assume $q<-\frac{m}{2}$.
The proof for $q>\frac{m}{2}$ is similar. Since $L_0$ acts on $M_{\alpha+q}$ with a nonzero scalar,
we can write $u=L_0v$ for some $v\in M_{\alpha+q}$. Then
$$(x\otimes t^{p-q})\otimes L_0v=\sum\limits_{i=0}^{m}(-1)^i\binom{m}{i}(x\otimes t^{p-q-i})\otimes L_iv
-\sum\limits_{i=1}^{m}(-1)^i\binom{m}{i}(x\otimes t^{p-q-i})\otimes L_iv.$$
Note that $I=\dot{\frak g}\otimes \mathcal{A}$ has a natural $\mathcal{A}$-module structure. Hence, we have
\begin{align*}
&t^r\Big(\sum\limits_{i=0}^{m}(-1)^i\binom{m}{i}(x\otimes t^{j-i}) L_{p+i}v\Big)=\sum\limits_{i=0}^{m}(-1)^i\binom{m}{i}(x\otimes t^{r+j-i})L_{p+i}v=0,\\
&t^r\xi\Big(\sum\limits_{i=0}^{m}(-1)^i\binom{m}{i}(x\otimes t^{j-i}) L_{p+i}v\Big)=\sum\limits_{i=0}^{m}(-1)^i\binom{m}{i}(x\otimes t^{r+j-i})\xi L_{p+i}v=0
\end{align*}
by Lemma \ref{lm:3.3}, which implies that
$\sum\limits_{i=0}^{m}(-1)^i\binom{m}{i}(x\otimes t^{j-i})\otimes L_{p+i}v \in K(M)$
for all $i, j, p\in\bZ, x\in\dot{\frak g}$ and $v\in M$.
Since $|q+i|\leq \frac{m}{2}(1\leq i\leq m)$, we have
$$\sum\limits_{i=1}^{m}(-1)^i\binom{m}{i}(x\otimes t^{p-q-i})\otimes L_iv
\in \sum\limits_{|k|\le\frac{m}{2}}\Big((x\otimes t^{p-k})\otimes M_{\alpha+k}+(x\otimes t^{p-k}\xi) \otimes M_{\alpha+k}\Big)+K(M)$$ by induction assumption.
So does $(x\otimes t^{p-q})\otimes L_0v$. Similarly, we have
$$(x\otimes t^{p-q}\xi)\otimes L_0v\in \sum\limits_{|k|\le\frac{m}{2}}\Big((x\otimes t^{p-k})\otimes M_{\alpha+k}+(x\otimes t^{p-k}\xi) \otimes M_{\alpha+k}\Big)+K(M).$$
Therefore, the lemma follows from the fact that $\dim M_{\alpha+k}<\infty$ for any fixed $k$ and $\dot{\frak g}$ is finite-dimensional.
\end{proof}

\section{Cuspidal modules of $\frak g$}
We denote by $U(L)$ be the universal enveloping algebra of the Lie (super)algebra $L$. Let ${\cI}$ be the left ideal of $\cU=U(\wt{\frak g})$ generated by $t^i\cdot t^j-t^{i+j}$, $t^0-1$, $t^i\cdot \xi-t^i\xi$ and $\xi\cdot\xi$ for all $i, j\in\mathbb{Z}$. Then it is clear that $\cI$ is an ideal of $\cU$. Now we have the quotient algebra $\ol{\cU}=\cU/\cI=(U(\frak g)U(\mathcal{A}))/\mathcal{I}$. From PBW theorem, we may identify $\mathcal{A}$, $\frak g$ with their images in $\ol{\cU}$. Thus $\ol{\cU}=\mathcal{A}\cdot U(\frak g)$. Then the category of $\mathcal{A}\frak g$-modules
is equivalent to the category of $\ol{\cU}$-modules.

For $i\in \bZ\backslash \{0\}, m\in \bZ$ and $x\in\dot{\frak g}$, let
\begin{align*}
&X_i=t^{-i}\cdot L_i+\frac{i}{2}t^{-i}\xi\cdot G_i-L_0,\\
&Y_i=t^{-i}\cdot G_i-2t^{-i}\xi\cdot L_i-G_0+2\xi\cdot L_0,\\
&\overline{X}_m^x=t^{-m}\cdot (x\otimes t^m)+mt^{-m}\xi\cdot(x\otimes t^m\xi),\\
&\overline{Y}_m^x=t^{-m}\xi\cdot (x\otimes t^m)-t^{-m}\cdot(x\otimes t^m\xi)
\end{align*}
be elements in $\ol{\cU}$ and $\cT$ be the subspace of $\ol{\cU}$ spanned by $\{X_i, Y_i, \overline{X}_m^x, \overline{Y}_m^x\mid i\in\bZ\backslash \{0\}, m\in\bZ, x\in\dot{\frak g}\}$.
In Proposition 3.1 of \cite{CLL}, the authors show that
\begin{align*}
[G_0, X_i]=[G_0, Y_i]=[t^n, X_i]=[t^n, Y_i]=[\xi, X_i]=[\xi, Y_i]=0,
\end{align*}
and
\begin{align*}
&[X_i, X_j]=-jX_j+iX_i+(j-i)X_{i+j},\\
&[X_i, Y_j]=-jY_j+\frac{i}{2}Y_i+(j-\frac{i}{2})Y_{i+j}, \\
&[Y_i, Y_j]=2(X_i+ X_j-X_{i+j})
\end{align*}
for $i,j\in\bZ\backslash \{0\}$. Just by simple calculation, we can get
$$[G_0, \overline{X}_m^x]=[G_0, \overline{Y}_m^x]=[t^n, \overline{X}_m^x]=[t^n, \overline{Y}_m^x]=[\xi, \overline{X}_m^x]=[\xi, \overline{Y}_m^x]=0, \forall m\in\bZ.$$
Now, we have
\begin{align*}
[X_i, \overline{X}_m^x]=&[t^{-i}\cdot L_{i}+\frac{i}{2}t^{-i}\xi\cdot G_{i}-L_{0},\ t^{-m}\cdot (x\otimes t^m)+mt^{-m}\xi\cdot (x\otimes t^m\xi)]\\
=&t^{-i}[L_{i},t^{-m}]\cdot (x\otimes t^m)+t^{-(i+m)}\cdot [L_{i},x\otimes t^m]+mt^{-i}[L_{i},t^{-m}\xi]\cdot (x\otimes t^m\xi)\\
&+mt^{-(i+m)}\xi\cdot [L_{i},x\otimes t^m\xi]+\frac{i}{2}t^{-i}\xi[G_{i},t^{-m}]\cdot (x\otimes t^m)\\
&+\frac{i}{2}t^{-(i+m)}\xi\cdot [G_{i},x\otimes t^m]+\frac{im}{2}t^{-i}\xi[G_{i},t^{-m}\xi]\cdot (x\otimes t^m\xi)-[L_{0},t^{-m}]\cdot (x\otimes t^m)\\
&-t^{-m}\cdot [L_{0},x\otimes t^m]-m[L_{0},t^{-m}\xi]\cdot (x\otimes t^m\xi)-mt^{-m}\xi\cdot [L_{0},x\otimes t^m\xi]\\
=&-mt^{-m}\cdot (x\otimes t^m)+mt^{-(i+m)}\cdot (x\otimes t^{i+m})\\
&+m(\frac{i}{2}-m)t^{-m}\xi\cdot (x\otimes t^{m}\xi)+m(\frac{i}{2}+m)t^{-(i+m)}\xi\cdot (x\otimes t^{i+m}\xi)\\
&+\frac{im}{2}t^{-(i+m)}\xi\cdot (x\otimes t^{i+m}\xi)-\frac{im}{2}t^{-m}\xi\cdot (x\otimes t^{m}\xi)\\
=&-m\overline{X}_m^x+m\overline{X}_{i+m}^x.
\end{align*}
Similarly, we can get
\begin{align*}
&[X_i, \overline{Y}_m^x]=-m\overline{Y}_m^x+(\frac{i}{2}+m)\overline{Y}_{i+m}^x,\\
&[Y_i, \overline{X}_m^x]=m\overline{Y}_m^x-m\overline{Y}_{i+m}^x, \ \ [Y_i, \overline{Y}_m^x]=-\overline{X}_m^x+\overline{X}_{i+m}^x,\\
&[\overline{X}_m^x,\overline{X}_n^y]=\overline{X}_{m+n}^{[x,y]},\ \ [\overline{X}_m^x,\overline{Y}_n^y]=\overline{Y}_{m+n}^{[x,y]},\ \ [\overline{Y}_m^x,\overline{Y}_n^y]=0
\end{align*}
for $i,j\in\bZ\backslash \{0\}, m,n\in\bZ$. So we have the following lemma.
\begin{lemm}\label{lm:4.1}
\begin{enumerate}
\item $[\cT,G_0]=[\cT, \mathcal{A}]=0$;
\item $\cT$ is a Lie super subalgebra of $\ol{\cU}$.
\end{enumerate}
\end{lemm}
\begin{prop}\label{prop:4.2}
The associative superalgebras $\ol{\cU}$ and $\cK\otimes U(\cT)$ are isomorphic, where $\cK$ is the Weyl superalgebra
$\mathcal{A}[\frac{\partial}{\partial t},\frac{\partial}{\partial \xi}]$.
\end{prop}
\begin{proof}
Note that $U(\cT)$ is an associative subalgebra of $\ol{\cU}$. Define the map $\tau: \mathcal{A}[G_0]\rightarrow \cK$ by
$\tau|_{\mathcal{A}}=\mbox{Id}_{\mathcal{A}}, \tau(G_0)=t\xi\frac{\partial}{\partial t}-\frac{\partial}{\partial \xi}.$
Obviously, we have
$$\tau(-t^{-1}G_0^2)=\frac{\partial}{\partial t},\ \tau(-G_0-\xi G_0^2)=\frac{\partial}{\partial \xi}, \ \tau(t^{\pm 1})=t^{\pm 1},\ \tau(\xi)=\xi$$
and $\tau$ is an isomorphism of associative superalgebras $\mathcal{A}[G_0]$ and $\cK$.

Then we define the map $\phi: \mathcal{A}[G_0]
\otimes U(\cT)\rightarrow \ol{\cU}$ by
$$\phi(t^i\xi^rG_0^k\otimes y)=t^i\xi^r\cdot G_0^k\cdot y+\cI,\ \forall i\in\bZ, r\in\{0,1\}, k\in \bZ_+, y\in U(\cT).$$
The restrictions of $\phi$ on $\mathcal{A}[G_0]$ and $U(\cT)$ are
well-defined homomorphisms of associative superalgebras. Note that $[\cT, G_0]=[\cT, \mathcal{A}]=0$, $\phi$ is a well-defined
homomorphism of associative superalgebras. From
\begin{align*}
&\phi(t^i\otimes X_i-\frac{i}{2}t^i\xi\otimes Y_i+t^iL_0\otimes 1-\frac{i}{2}t^i\xi G_0\otimes 1)=L_i,\\
&\phi(t^i\otimes Y_i+2t^i\xi\otimes X_i+t^iG_0\otimes 1)=G_i, \\
&\phi(t^i\otimes \overline{X}_i^x+it^i\xi\otimes \overline{Y}_i^x)=x\otimes t^i, \\
&\phi(t^i\xi\otimes\overline{X}_i^x-t^i\otimes \overline{Y}_i^x)=x\otimes t^i\xi,
\end{align*}
we can see that $\phi$ is surjective.

By PBW theorem, we know that $\ol{\cU}$ has a basis of monomials in variables
$\{L_i, G_i, x\otimes t^m, x\otimes t^m\xi\mid i\in \bZ\setminus\{0\}, m\in \bZ, x\in\dot{\frak g}\}$ over $\mathcal{A}[G_0]$. Therefore $\ol{\cU}$
has an $\mathcal{A}[G_0]$-baisis consisting of monomials in the variables
$\{t^{-i}\cdot L_i-L_0,\ t^{-i}\cdot G_i-G_0,\ \overline{X}_m^x, \overline{Y}_m^x\mid i\in \bZ\setminus\{0\}, m\in \bZ, x\in\dot{\frak g}\}$.
So $\phi$ is injective and hence an isomorphism.
\end{proof}

For the Weyl superalgebra $\cK$ and any $\lambda\in\bC$, let $\sigma_\lambda$ be the automorphism of $\cK$ with
$\sigma_\lambda(t\frac{d}{dt})=t\frac{d}{dt}+\lambda, \sigma_\lambda(\frac{\partial}{\partial \xi})=\frac{\partial}{\partial \xi}, \sigma_\lambda|_\mathcal{A}=\mbox{id}_\mathcal{A}$.
Denote $\mathcal{A}(\lambda)=\mathcal{A}^{\sigma_\lambda}$. It is clear that $\mathcal{A}(\lambda)\cong \cK/\cI_\lambda$, where $\cI_\lambda$ is the left ideal of $\cK$
generated by $t\frac{d}{dt}-\lambda$ and $\frac{\partial}{\partial \xi}$. We need the following lemma.

\begin{lemm}\cite[Lemma 3.5]{XL2}\label{lm:4.3}
Any simple weight $\cK$-module is isomorphic to some $\mathcal{A}(\lambda)$ for some $\lambda\in \bC$ up to a parity-change, where
$\mathcal{A}(\lambda)$ is a strictly simple $\cK$-module.
\end{lemm}

Let $\fm$ be the maximal ideal of $\mathcal{A}$ generated by $t-1$ and $\xi$, and $\fm^k\mathcal{A}$ is spanned by the set
$$\{(t-1)^kt^i,(t-1)^kt^i\xi, (t-1)^{k-1}t^i\xi\mid i\in \bZ, k\in\mathbb{N}\}.$$
\begin{lemm}\label{lm:4.4}
Let $k,l\in\bZ_+$. Then for all $i,j\in\bZ$ and $x\in\dot{\frak g}$, there is
\begin{align*}
&[(t-1)^kL_i,(t-1)^l L_j]=(l-k+j-i)(t-1)^{k+l}L_{i+j}+(l-k)(t-1)^{k+l-1}L_{i+j},\\
&[(t-1)^kL_i,(t-1)^l G_j]=(j-\frac{i}{2})(t-1)^{k+l}G_{i+j}+(l-\frac{k}{2})(t-1)^{k+l-1}G_{i+j+1},\\
&[(t-1)^kG_i,(t-1)^l G_j]=-2(t-1)^{k+l}L_{i+j},\\
&[(t-1)^kL_i,x\otimes(t-1)^l t^j]=jx\otimes(t-1)^{k+l}t^{i+j}+lx\otimes(t-1)^{k+l-1}t^{i+j+1},\\
&[(t-1)^kL_i,x\otimes(t-1)^l t^j\xi]=(j+\frac{i}{2})x\otimes(t-1)^{k+l}t^{i+j}\xi+(l+\frac{k}{2})x\otimes(t-1)^{k+l-1}t^{i+j+1}\xi,\\
&[(t-1)^kG_i,x\otimes(t-1)^l t^j]=jx\otimes(t-1)^{k+l}t^{i+j}\xi+lx\otimes(t-1)^{k+l-1}t^{i+j+1}\xi,\\
&[(t-1)^kG_i,x\otimes(t-1)^l t^j\xi]=-x\otimes(t-1)^{k+l}t^{i+j}.
\end{align*}
\end{lemm}
\begin{proof}
The results follow from direct computations.
\end{proof}
\begin{lemm}\label{lm:4.5}
For $k\in\bZ_+$, let $\fa_k=(t-1)^{k+1}{\overline {\frak s}}\ltimes(\dot{\frak g}\otimes{\fm}^k\mathcal{A})$. Then
\begin{enumerate}
\item[(1)] $\fa_0$ is a Lie super subalgebra of $\frak g$;
\item[(2)] $\fa_k$ is an ideal of $\fa_0$;
\item[(3)] $[\fa_1,\fa_k]\subseteq\fa_{k+1}$;
\item[(4)] The ideal of $\fa_0$ generated by $(t-1)^{k}\mathcal{W}$ contains $\fa_{k}$.
\end{enumerate}
\end{lemm}
\begin{proof}
Note that $\fa_0=\mbox{span}\{(t-1)L_i, (t-1)G_i, x\otimes t^m, x\otimes t^m\xi\mid i\in \bZ, m\in \bZ, x\in\dot{\frak g}\}$ and
\begin{eqnarray*}
&\fa_k=\mbox{span}\{(t-1)^{k+1}L_i, (t-1)^{k+1}G_i, x\otimes (t-1)^kt^i, x\otimes (t-1)^kt^i\xi, x\otimes (t-1)^{k-1}t^i\xi\\
&\mid i\in\bZ, k\in\mathbb{N}, x\in\dot{\frak g}\}.
\end{eqnarray*}
Therefore, the conclusions $(1)-(3)$ can be verified directly by using Lemma \ref{lm:4.4}.

The ideal generated by $(t-1)^{k}L_i$ in $\fa_0$ is denoted by $\frak b$. Since
$$[(t-1)^kL_i, (t-1)L_0]-[(t-1)^kL_0, (t-1)L_i]=-2i(t-1)^{k+1}L_i,$$
we have $(t-1)^{k+1}L_i\in\frak b$ for $i\neq 0$. Furthermore, since
$$[(t-1)^kL_1, (t-1)L_{-1}]-[(t-1)^kL_{-1}, (t-1)L_1]=-4(t-1)^{k+1}L_0\in\frak b,$$
there is $(t-1)^{k+1}L_i\in\frak b$ for $i\in\bZ$. Similarly, we can get $(t-1)^{k+1}G_i, x\otimes (t-1)^kt^i, x\otimes (t-1)^kt^i\xi, x\otimes (t-1)^{k-1}t^i\xi\in\frak b$.
Hence, $\fa_{k}\subseteq \frak b$ and the conclusion (4) holds.
\end{proof}
\begin{prop}
The Lie superalgebras $\cT$ and $\fa_0$ are isomorphic.
\end{prop}
\begin{proof}
Since
\begin{align*}
&(t-1)L_i=L_{i+1}-L_i=(L_{i+1}-L_0)-(L_i-L_0),\\
&(t-1)G_i=G_{i+1}-G_i=(G_{i+1}-G_0)-(G_i-G_0),
\end{align*}
we have $\fa_0=(t-1){\overline {\frak s}}\ltimes(\dot{\frak g}\otimes \mathcal{A})$ is spanned by the set
$\{L_i-L_0, G_i-G_0, x\otimes t^m, x\otimes t^m\xi\mid i\in \bZ\setminus\{0\}, m\in \bZ, x\in\dot{\frak g}\}$.
It is easy to verify that the linear map $\varphi: \cT\rightarrow \fa_0$ defined by
\begin{align*}
&\varphi(X_i)=L_i-L_0,\ \ \varphi(Y_i)=G_i-G_0,\\
&\varphi(\overline{X}_m^x)=x\otimes t^m, \ \ \varphi(\overline{Y}_m^x)=-x\otimes t^m\xi
\end{align*}
is a Lie superalgebra isomorphism.
\end{proof}

For any $\fa_0$-module $V$, we have the $\mathcal{A}\frak g$-module $\Gamma(\lambda, V)=(\mathcal{A}(\lambda)\otimes V)^{\varphi_1}$, where
$\varphi_1: \ol{\cU}\stackrel{\phi^{-1}}{\longrightarrow}\cK\otimes U(\cT)\stackrel{1\otimes \varphi}{\longrightarrow} \cK\otimes U(\fa_0)$. More precisely, we give a kind of cuspidal module as follows.
\begin{defi}\label{def:4.7}
Let $V$ be an $\fa_0$-module $V$, then
$\Gamma(\lambda, V):=\mathcal{A}\otimes V$ becomes an $\mathcal{A}\frak g$-module with actions
\begin{align*}
&t^i\xi^r\circ(f\otimes v)=t^i\xi^rf\otimes v,\\
&L_i\circ(f\otimes v)=t^if\otimes(L_i-L_0)\centerdot v-(-1)^{|f|}\frac{i}{2}t^i\xi f\otimes(G_i-G_0)\centerdot v\\
&\quad\quad\quad\quad\quad\quad+t^i(\lambda f+t\frac{\partial}{\partial t}(f))\otimes v+\frac{i}{2}t^i\xi\frac{\partial}{\partial\xi}(f)\otimes v,\\
&G_i\circ(f\otimes v)=(-1)^{|f|}t^if\otimes(G_i-G_0)\centerdot v+2t^i\xi f\otimes(L_i-L_0)\centerdot v\\
&\quad\quad\quad\quad\quad\quad+t^i\xi(\lambda f+t\frac{\partial}{\partial t}(f))\otimes v-t^i\frac{\partial}{\partial\xi}(f)\otimes v,\\
&(x\otimes t^i)\circ(f\otimes v)=t^if\otimes (x\otimes t^i)\centerdot v-(-1)^{|f|}it^i\xi f\otimes (x\otimes t^i\xi)\centerdot v,\\
&(x\otimes t^i\xi)\circ(f\otimes v)=t^i\xi f\otimes (x\otimes t^i)\centerdot v+(-1)^{|f|}t^if\otimes (x\otimes t^i\xi)\centerdot v
\end{align*}
for any $f\in \mathcal{A}, v\in V, i\in\bZ$ and $r\in\{0,1 \}$, where the symbol $\centerdot$ represents the action of $\fa_0$ on $V$.
\end{defi}

\begin{lemm}\label{lm:4.8}
1. For any $\lambda\in \bC$ and any simple $\fa_0$-module $V$, $\Gamma(\lambda, V)$ is a simple weight $\mathcal{A}\frak g$-module.

2. Up to a parity change, any simple weight $\mathcal{A}\frak g$-module $M$ is isomorphic to some $\Gamma(\lambda, V)$ for some $\lambda\in \mbox{supp}(M)$ and
some simple $\fa_0$-module $V$.
\end{lemm}
\begin{proof}
1. From Lemmas \ref{lm:2.1} and \ref{lm:4.3}, we know that $\mathcal{A}(\lambda)\otimes V$ is a simple $\cK\otimes U(\cT)$-module for any
$\lambda\in \bC$ and any simple $\fa_0$-module $V$. From the definition of $\Gamma(\lambda, V)$, we have the first statement.

2. Let $M$ be any simple weight $\mathcal{A}\frak g$-module with $\lambda\in \mbox{supp}(M)$. Then $M$ is a simple $\ol{\cU}$-module and $M^{\varphi_1^{-1}}$
is a simple $\cK\otimes U(\fa_0)$-module.
Fix a nonzero homogeneous element $v\in (M^{\varphi_1^{-1}})_\lambda$. Since $V'=\bC[\frac{\partial}{\partial\xi}]v$ is a finite-dimensional
super subspace with $\frac{\partial}{\partial\xi}$ acting nilpotently, we may find a nonzero homogeneous element $v'\in V'$ with $\cI_\lambda v'=0$.
Clearly, $\cK v'$ is isomorphic to $\mathcal{A}(\lambda)$ or $\Pi(\mathcal{A}(\lambda))$. From Lemma \ref{lm:2.1}, there exists a simple $U(\fa_0)$-module $V$ such
that $M^{\varphi_1^{-1}}\cong \mathcal{A}(\lambda)\otimes V$ or $M^{\varphi_1^{-1}}\cong \Pi(\mathcal{A}(\lambda))\otimes V$. So this conclusion holds.
\end{proof}

Thus, to classify all simple weight $\mathcal{A}\frak g$-modules, it suffices to classify all simple $\fa_0$-modules. In particular, to classify all simple
cuspidal $\mathcal{A}\frak g$-modules, it suffices to classify all finite-dimensional simple $\fa_0$-modules. Now we introduce two known conclusions that we are going to use.
\begin{lemm}\cite[Lemma 2.6]{CLW}\label{lm:4.9}
Any co-finite ideal of $(t-1)\mathcal{W}$ contains $(t-1)^k\mathcal{W}$ for large $k$.
\end{lemm}

\begin{lemm}\cite[Theorem 2.1]{Mo}\label{lm:4.10}
Let $V$ be a finite-dimensional module for the Lie superalgebra $L=L_{\bar{0}}\oplus L_{\bar{1}}$
such that the elements of $L_{\bar{0}}$ and $L_{\bar{1}}$ respectively are nilpotent endmorphisms of $V$. Then there exists a nonzero element $v\in V$
such that $xv=0$ for all $x\in L$.
\end{lemm}
\begin{lemm}\label{lm:4.11}
1. Let $V$ be any finite-dimensional $\fa_0$-module. Then there exists $k\in \bN$ such that $\fa_kV=0$.

2. Let $V$ be any finite-dimensional simple $\fa_0$-module. Then $\fa_1V=0$.
\end{lemm}
\begin{proof}
1. Let $V$ be any finite-dimensional $\fa_0$-module. Then $V$ is a finite-dimensional $(t-1)\mathcal{W}$-module. By Lemma \ref{lm:4.9}, there exsits
$k\in \bZ$ such that $(t-1)^k\mathcal{W}V=0$. So the first statement follows from Lemma \ref{lm:4.5}.

2.  Let $V$ be any finite-dimensional simple $\fa_0$-module. Then $V$ is a simple finite-dimensional   $\fa_0/\mbox{ann}(V)$-module, where
$\mbox{ann}(V)$ is the ideal of $\fa_0$ that annihilates $V$ and $\fa_k\subseteq \mbox{ann}(V)$ for some $k\in\bN$. So $V$ is a finite-dimensional
module for $(\fa_0)_{\bar{0}}+\mbox{ann}(V)$. Since $[\fa_1,\fa_k]\subseteq\fa_{k+1}$ in Lemma \ref{lm:4.5}, we have
$$((\fa_1)_{\bar{0}}+\mbox{ann}(V))^{k-1}\subseteq (\fa_k)_{\bar{0}}+\mbox{ann}(V)=\mbox{ann}(V),$$
which implies that $(\fa_1)_{\bar{0}}+\mbox{ann}(V)$ acts nilpotently on $V$. Since $[x, x]\in (\fa_1)_{\bar{0}}+\mbox{ann}(V)$ for any
$x\in (\fa_1)_{\bar{1}}+\mbox{ann}(V)$, every element in $(\fa_1)_{\bar{1}}+\mbox{ann}(V)$ acts nilpotently on $V$. Hence, by Lemma \ref{lm:4.10}, there is
a nonzero element $v\in V$ annihilated by $\fa_1+\mbox{ann}(V)$.

Let $V'=\{v\in V|xv=0,\forall x\in \fa_1\}$. So $V'\neq \emptyset$. For any $y\in \fa_0, x\in \fa_1$ and $v\in V'$, there is
$xyv=(-1)^{|x||y|}yxv+[x, y]v=0$, which implies that $yv\in V'$. Hence $V'=V$ by the simplicity of $V$. And therefore $\fa_1V=0$.
\end{proof}

\begin{lemm}\cite[Lemma 2.4]{CLL}\label{lm:4.12}
Let $R=(t-1){\overline {\frak s}}/(t-1)^{2}{\overline {\frak s}}$. Then $R$ is a two dimensional Lie superalgebra with $|X|=\bar{0}, |Y|=\bar{1}$ and nontrivial
brackets $[X, Y]=\frac{1}{2}Y$.
\end{lemm}

In fact, we have $(t-1)^2L_i\equiv 0$ and $(t-1)^2G_i\equiv 0$ in the Lie superalgebra $R$. Then we have
$$t(t-1)L_i\equiv (t-1)L_i,\ t(t-1)G_i\equiv (t-1)G_i,\ \forall i\in \bZ, $$
which implies that $X=(t-1)L_0$ and $Y=(t-1)G_0$.
\begin{theo}\label{th:4.13}
Up to a parity change, any simple cuspidal $\mathcal{A}\frak g$-module is isomorphic to a tensor module $\Gamma(\lambda, V)$ for some finite-dimensional simple $R\ltimes\dot{\frak g}$-module $V$ and some $\lambda\in \bC$.
\end{theo}
\begin{proof}
Let $V$ be any finite-dimensional simple $\fa_0$-module. We know that $V$ is a finite-dimensional simple $\fa_0/\fa_1$-module by Lemma \ref{lm:4.11}. Note that
$$\fa_0/\fa_1\cong \big((t-1)\overline {\frak s}/ (t-1)^2\overline {\frak s}\big)\ltimes \big(\dot{\frak g}\otimes (\mathcal{A}/\fm\mathcal{A})\big)\cong R\ltimes\dot{\frak g}$$
by Lemma \ref{lm:4.12}. Therefore, from Lemma \ref{lm:4.8}, the conclusion is true.
\end{proof}
So far, we know that, to classify all simple cuspidal $\mathcal{A}\frak g$-modules, it suffices to classify all finite-dimensional simple modules over $R\ltimes\dot{\frak g}$. Since any simple finite-dimensional  module over two-dimensional Lie superalgebra $R$ is one-dimensional with $X\cdot u=au,\ Y\cdot u=0$ for some $a\in\bC$, there is the following lemma.
\begin{lemm}\cite[Corollary 3.5]{CLL}\label{lm:4.14}
Up to a parity change, any simple cuspidal $\mathcal{A}\overline{\frak s}$-module is isomorphic to some $\Gamma(\lambda, a)=\mathcal{A}\otimes \bC u$ with $\lambda, a\in\bC$
defined as follows:
\begin{align*}
&t^i\xi^r\circ(f\otimes u)=t^i\xi^rf\otimes u,\\
&L_i\circ(t^m\otimes u)=(\lambda+m+ia)t^{i+m}\otimes u,\\
 &L_i\circ(t^m\xi\otimes u)=(\lambda+m+ia+\frac{i}{2})t^{i+m}\xi\otimes u,\\
&G_i\circ(t^m\otimes u)=(\lambda+m+2ia)t^{i+m}\xi\otimes u,\\
 &G_i\circ(t^m\xi\otimes u)=-t^{i+m}\otimes u,
\end{align*}
where $i, m\in \bZ, f\in \mathcal{A}, r\in\{0,1\}$ and $|u|=\bar{0}$.
\end{lemm}
\begin{lemm}\label{lm:4.15}
Any finite-dimensional simple $R\ltimes\dot{\frak g}$-module $V$ is induced by the finite-dimensional simple $\dot{\frak g}$-module $V$ with $X\centerdot v=bv, Y\centerdot v=0, \ \forall v\in V$ for some $b\in\mathbb C$.
\end{lemm}
\begin{proof}
Let $V$ be a finite-dimensional simple $R\ltimes\dot{\frak g}$-module and $W=\{v\in V\mid Y\centerdot v=0\}$. We have $V=W$ because $V$ is simple and $W$ is a nontrivial submodule of $V$. Since $X=(t-1)L_0$, we get $[X, \dot{\frak g}]=0$ from Lemma \ref{lm:4.4}. So $X\centerdot v=bv,\ \forall v\in V$ for some $b\in\mathbb C$ by Schur's Lemma.
\end{proof}
For a finite-dimensional simple $\dot{\frak g}$-module $V$, from Lemma \ref{lm:4.15}, there exists $b\in\bC$ such that
$$(L_i-L_0)\centerdot v=\big((L_i-L_{i-1})+\ldots+(L_1-L_0)\big)\centerdot v=iX\centerdot v=ibv, \ \forall v\in V.$$
Similarly, we get $(G_i-G_0)\centerdot v=iY\centerdot v=0$. Therefore, the $\mathcal{A}\frak g$-module $\mathcal{A}\otimes V$ in Definition \ref{def:4.7} can be expressed as $\Gamma(\lambda, V, b)$ with actions
\begin{align*}
&t^i\xi^r\circ(f\otimes v)=t^i\xi^rf\otimes v,\\
&L_i\circ(t^m\xi^r\otimes v)=(\lambda+m+ib+\frac{i}{2}\delta_{\bar{1},\bar{r}})t^{i+m}\xi^r\otimes v,\\
&G_i\circ(t^m\otimes v)=(\lambda+m+2ib)t^{i+m}\xi\otimes v,\\
&G_i\circ(t^m\xi\otimes v)=-t^{i+m}\otimes v,\\
&(x\otimes t^i)\circ(f\otimes v)=t^if\otimes (x\centerdot v),\\
&(x\otimes t^i\xi)\circ(f\otimes v)=t^i\xi f\otimes (x\centerdot v),
\end{align*}
where $i, m\in\bZ, f\in \mathcal{A}, r\in\{0,1\}$ and $v\in V$.

Below we can present the main conclusion of this section.
\begin{theo}\label{th:4.16}
Up to a parity change, any simple cuspidal $\frak g$-module is isomorphic to a simple quotient of a tensor module $\Gamma(\lambda, V, b)$
for some finite-dimensional simple $\dot{\frak g}$-module $V$ and $\lambda, b\in\bC$.
\end{theo}
\begin{proof}
Let $M$ be a simple cuspidal $\frak g$-module. If $M$ is a trival $\frak g$-module, then $M$ is a simple
quotient of the simple cuspidal $\mathcal{A}\frak g$-module $\mathcal{A}\otimes\bC$ with $\bC$ a trivial $\dot{\frak g}$-module. Now suppose $M$ is a nontrivial simple cuspidal $\frak g$-module.

If $IM=0$, then $M$ is a simple cuspidal $\overline{\frak s}$-module. Hence, $M$ is a simple quotient of a simple cuspidal $\mathcal{A}\overline{\frak s}$-module
$\Gamma(\lambda, a)$ by Theorem 3.9 in \cite{CLL}. Since $\frak g=\overline {\frak s}\ltimes I$, any $\mathcal{A}\overline{\frak s}$-module is naturally an  $\mathcal{A}\frak g$-module with trivial $I$-action.

If $IM\neq 0$, then $IM=M$ since $M$ is simple. So there is an epimorphism $\pi: \widehat{M}\rightarrow M$. From Proposition \ref{prop:3.5}, $\widehat{M}$ is a cuspidal $\mathcal{A}\frak g$-module. Consider the composition series of $\mathcal{A}\frak g$-submodules in $\widehat{M}$
$$0=\widehat{M}^{(1)}\subset \widehat{M}^{(2)}\subset\cdots \subset \widehat{M}^{(s)}=\widehat{M}$$
with the quotients $\widehat{M}^{(i)}/\widehat{M}^{(i-1)}$ being simple cuspidal $\mathcal{A}\frak g$-modules. Let $l$ be the
minimal integer such that $\pi(\widehat{M}^{(l)})\ne 0$.  Since $M$ is simple $\frak g$-module, we have $\pi(\widehat{M}^{(l)})=M$
and $\pi(\widehat{M}^{(l-1)})=0$. This induces an epimorphism of $\frak g$ modules from $\widehat{M}^{(l)}/\widehat{M}^{(l-1)}$ to $M$.
Hence, $M$ is isomorphic to a simple quotient of cuspidal $\mathcal{A}\frak g$-module $\widehat{M}^{(l)}/\widehat{M}^{(l-1)}$.
We can complete the proof by Theorem \ref{th:4.13}.
\end{proof}

\section{Main result}
Through the previous study, it remains to classify all simple Harish-Chandra modules over $\widehat{\frak g}$ which is not cuspidal.
Let $\{x_s\mid s=1,2,\cdots, l\}$ be a basis of Lie algebra $\dot{\frak g}$. Then $\mbox{dim}\,\widehat{\frak g}_n=2+2l$ for $n\in\mathbb{Z}$
and $n\neq 0$.
The following result is well known.
\begin{lemm}\cite[Lemma 1.6]{Mu}\label{lm:5.1}
Let $V$ be a weight module with finite-dimensional weight spaces for the Virasoro algebra with $\mbox{supp}(V)\subseteq \lambda+\mathbb{Z}$.
If for any $v\in V$, there exists $N(v)\in \mathbb{N}$ such that $L_iv=0, \forall i\geq N(v)$, then $\mbox{supp}(V)$ is upper bounded.
\end{lemm}
\begin{lemm}\label{lm:5.2}
Suppose $M$ is a simple Harish-Chandra module over $\widehat{\frak g}$ which is not cuspidal, then $M$ is a highest (or lowest) weight module.
\end{lemm}
\begin{proof}
For a fixed $\lambda\in\mbox{supp}(M)$, there is a $k\in\mathbb{Z}$ such that $\mbox{dim}M_{\lambda-k}>(2+2l)\mbox{dim}M_\lambda+2\mbox{dim}M_{\lambda+1}$
since $M$ is not cuspidal. Without loss of generality, we may assume that $k\in \mathbb{N}$. Then there exists a nonzero element $\omega\in M_{\lambda-k}$
such that
$$L_k\omega=L_{k+1}\omega=G_k\omega=G_{k+1}\omega=0$$
and
$$(x_s\otimes t^k)\omega=(x_s\otimes t^k\xi)\omega=0,$$
where $s\in\{1,2,\cdots,l\}$. Therefore, we get $L_p\omega=G_p\omega=(x_s\otimes t^p)\omega=(x_s\otimes t^p\xi)\omega=0$ for all $p\geq k^2$
since $[\widehat{\frak g}_i, \widehat{\frak g}_j]=\widehat{\frak g}_{i+j}$ for $j\neq 0$.

It is easy to see that $M'=\{v\in M\mid\mbox{dim}\ \widehat{\frak g}_+v<\infty\}$ is a nonzero submodule of $M$. So $M=M'$ by the simplicity of $M$.
Since $M$ is a weight module over the Virasoro algebra, we have $\mbox{supp}(M)$ is upper bounded by Lemma \ref{lm:5.1}, that is $M$ is a highest weight module.
\end{proof}
To sum up, we have the following result.
\begin{theo}$(${\bf Main theorem}$)$\label{main}
Let $M$ be a simple Harish-Chandra module over the superconformal current algebra $\widehat{\frak g}$.
Then $M$ is a highest weight module, a lowest weight module, or a simple quotient of a tensor module $\Gamma(\lambda, V, b)$ or $\Pi(\Gamma(\lambda, V, b))$
for some finite-dimensional simple $\dot{\frak g}$-module $V$ and $\lambda, b\in\bC$.
\end{theo}

\section{Application to the $N=1$ Heisenberg-Virasoro algebra}

As an application, we can directly get the classification results of the simple Harish-Chandra modules over the $N=1$ Heisenberg-Virasoro algebra.
\begin{defi}\cite{AJR}
The $N=1$ Heisenberg-Virasoro algebra  $\mathcal {S}$ is a Lie superalgebra with basis $\{L_i, H_i, G_i, Q_i,
C_L, C_{L,\alpha}, C_\alpha\mid i\in \mathbb{Z}\}$ and brackets
\begin{align*}
&[L_i,L_j]=(j-i)L_{i+j}+\delta_{i+j,0}\frac{1}{12}(i^3-i)C_L, \\
&[L_i,H_j]=jH_{i+j}+\delta_{i+j,0}(i^2+i)C_{L,\alpha},\\
&[L_i,G_j]=(j-\frac{i}{2})G_{i+j},\ \ [L_i,Q_j]=(j+\frac{i}{2})Q_{i+j},\\
&[H_i, H_j]=i\delta_{i+j,0}C_\alpha,\\
&[G_i,H_j]=-jQ_{i+j}, \ \ [H_i, Q_j]=0,\\
&[G_i,Q_j]=H_{i+j}+(2i-1)\delta_{i+j,0}C_{L,\alpha},\\
&[G_i,G_j]=-2L_{i+j}+\delta_{i+j,0}\frac{1}{12}(4i^2-1)C_L,\\
&[Q_i,Q_j]=\delta_{i+j,0}C_\alpha,
\end{align*}
where $\{L_i, H_i\mid i\in\bZ\}$ are the even generators, $\{G_i, Q_i\mid i\in\bZ\}$ are odd generators and $C_L, C_{L,\alpha}, C_\alpha$ are central elements.
\end{defi}

By direct calculation we can get that $\mathcal {S}$ is the universal central extension of $\overline{\mathcal {S}}=\mbox{span}\{L_i, H_i, G_i, Q_i\mid i\in\bZ\}$. Let $M$ be a simple cuspidal $\mathcal {S}$-module. Similar to Example 6.1 in \cite{CLW}, we have that $C_LM=C_{L,\alpha}M=C_\alpha M=0$. Thus, the classification of simple cuspidal $\mathcal {S}$-modules is equivalent to the classification of simple cuspidal $\overline{\mathcal {S}}$-modules.

Now, let $H_i=Z\otimes t^i$ and $Q_i=-Z\otimes t^i\xi$. It is easy to see that $\overline{\mathcal {S}}$ is a special case of
$\frak g=\overline{\frak s}\ltimes(\dot{\frak g}\otimes \mathcal{A})$ by taking $\dot{\frak g}=\bC Z$. Therefore, we have $R\ltimes\bC Z$ is a three-dimensional
Lie superalgebra with $(R\ltimes\bC Z)_{\bar{0}}=\mbox{span}\{X,Z\}$, $(R\ltimes\bC Z)_{\bar{1}}=\mbox{span}\{Y\}$ and the following nonzero brackets
$$[X,Y]=\frac{1}{2}Y$$
from Lemma \ref{lm:4.12}. We have the following Lemma.
\begin{lemm}
Any finite-dimensional simple module over the Lie superalgebra $R\ltimes\bC Z$ is one-dimensional with
$$X\centerdot v=bv,\ Y\centerdot v=0,\ Z\centerdot v=dv$$
for some $b,d\in\bC$.
\end{lemm}
Hence, any simple cuspidal $\mathcal {A}\overline{\mathcal {S}}$-module is isomorphic to some $\Gamma(\lambda,b,d)=\mathcal {A}\otimes \bC v$ with
$\lambda,b,d\in\bC$ up to a parity change from Theorem \ref{th:4.13}. The actions of $\mathcal {A}$ and $\overline{\mathcal {S}}$ on $\Gamma(\lambda,b,d)$ are as follows:
\begin{align*}
&t^i\xi^r\circ(f\otimes v)=t^i\xi^rf\otimes v,\\
&L_i\circ(t^m\otimes v)=(\lambda+m+ib)t^{i+m}\otimes v,\\
 &L_i\circ(t^m\xi\otimes v)=(\lambda+m+ib+\frac{i}{2})t^{i+m}\xi\otimes v,\\
&G_i\circ(t^m\otimes v)=(\lambda+m+2ib)t^{i+m}\xi\otimes v,\\
& G_i\circ(t^m\xi\otimes v)=-t^{i+m}\otimes v,\\
&H_i\circ (t^m\otimes v)=dt^{i+m}\otimes v,\\
& H_i\circ (t^m\xi\otimes v)=dt^{i+m}\xi\otimes v,\\
&Q_i\circ (t^m\otimes v)=-dt^{i+m}\xi\otimes v,\\
& Q_i\circ (t^m\xi\otimes v)=0,
\end{align*}
where $i, m\in \bZ, f\in \mathcal{A}, r\in\{0,1\}$ and $|v|=\bar{0}$.

Furthermore, from Theorem \ref{th:4.16}, any simple cuspidal $\overline{\mathcal {S}}$-module is isomorphic to a simple quotient of tensor module $\Gamma(\lambda,b,d)=\mathcal {A}\otimes \bC v$ with $\lambda,b,d\in\bC$. It is straightforward to check that as $\overline{\mathcal {S}}$-module, $\Gamma(\lambda,b,d)$ has a unique
nontrivial sub-quotient which we denote by $\Gamma'(\lambda,b,d)$. More precisely, we have the following lemma.
\begin{lemm}
1. As $\overline{\mathcal {S}}$-module, $\Gamma(\lambda,b,d)$ is simple if and only if one of the following three conditions is satisfied
$\ (1)d\neq 0;\ (2)\lambda\notin \bZ; \ (3)\lambda\in \bZ, b\neq 0, \frac{1}{2}.$\\
2. $\Gamma(\lambda_1,b_1,d_1)\cong \Gamma(\lambda_2,b_2,d_2)$ if and only if $\lambda_1-\lambda_2\in\bZ, b_1=b_2, d_1=d_2$ or
$\lambda_1\notin\bZ, \lambda_1-\lambda_2\in\bZ, b_1=\frac{1}{2}, b_2=0, d_1=d_2=0$ or $\lambda_1\notin\bZ, \lambda_1-\lambda_2\in\bZ, b_1=0, b_2=\frac{1}{2}, d_1=d_2=0$.\\
3. $\Gamma'(0,0,0)\cong \Gamma'(0,\frac{1}{2},0)$, where $\Gamma'(0,0,0)=\Gamma(0,0,0)/\{\bC(1\otimes v)\}$,
$\Gamma'(0,\frac{1}{2},0)=\mbox{span}\{t^m\otimes v, t^k\xi\otimes v\mid m,k\in\bZ, k\neq 0\}$.
\end{lemm}
\begin{proof}
1. Since
$$H_i\circ (t^m\otimes v)=dt^{i+m}\otimes v,\ H_i\circ (t^m\xi\otimes v)=dt^{i+m}\xi\otimes v,$$
it is easy to see that $\Gamma(\lambda,b,d)$ is simple if $d\neq 0$. Note that $\overline{\mathcal {S}}=\overline{\frak s}\ltimes \mathcal A$, so $\Gamma(\lambda,b,d)$ is also an $\overline{\frak s}$-module. We have that $\Gamma(\lambda,b,d)$ is simple if $\lambda\notin \bZ$ or $\lambda\in \bZ, b\neq 0, \frac{1}{2}$ from Lemma 3.10 in \cite{CLL}.

2. Suppose $\Gamma(\lambda_1,b_1,d_1)\cong \Gamma(\lambda_2,b_2,d_2)$, we have $d_1=d_2$ obviously. Then this conclusion holds from Lemma 3.10 in \cite{CLL}.

3. This conclusion can be verified directly.
\end{proof}

At this point, for the simple Harish-Chandra module over $N=1$ Heisenberg-Virasoro algebra $\mathcal {S}$, we have the following conclusion.
\begin{theo}\label{main1}
Let $M$ be a simple Harish-Chandra module over $\mathcal {S}$. Then $M$ is a highest weight module, a lowest weight module, or $\Gamma'(\lambda,b,d), \Pi(\Gamma'(\lambda,b,d))$
for some $\lambda,b,d\in \bC$.
\end{theo}
\noindent {\bf Acknowledgement} The authors are grateful to the referee for insightful comments and
suggestions. This work was supported by NSF of China (Grant Nos. 12101082, 12071405 and 12271383) and Jiangsu Provincial Double-Innovation Doctor Program (Grant No. JSSCBS20210742).


\section*{References}
\begin{thebibliography}{00}
\bibitem{AJR} D. Adamovi\'{c}, B. Jandri\'{c}, G. Radobolja, On the $N=1$ super Heisenberg-Virasoro vertex algebra, Contemp. Math., 768(2021), 167-178.

\bibitem{B} Y. Billig, Jet modules, Canad. J. Math., 59(2007), no. 4, 721-729.

\bibitem{BF} Y. Billig, V. Futorny, Classification of irreducible representations of Lie algebra of vector fields on a torus, J. Reine Angew. Math., 720(2016), 199-216.

\bibitem{BFI} Y. Billig, V. Futorny, K. Iohara, Classification of simple strong Harish-Chandra $W(m,n)$-modules, arXiv:2006.05618.

\bibitem{CLL} Y. Cai, D. Liu, R. L\"{u}, Classification of simple Harish-Chandra modules over the $N=1$ Ramond algebra, J. Algebra, 567(2021), 114-127.

\bibitem{CL} Y. Cai, R. L\"{u}, Classification of simple Harish-Chandra modules over the Neveu-Schwarz algebra and its contact subalgebra, J. Pure Appl. Algebra, 226(2022), 106866.

\bibitem{CLW} Y. Cai, R. L\"{u}, Y. Wang, Classification of simple Harish-Chandra modules for map (super)algebras related to the Virasoro algebra, J. Algebra, 570(2021), 397-415.

\bibitem{CHL} B. Chen, P. Hao, Y. Liu, Supersymmetric warped conformal field theory, Physical Review D, 102(2020), 065016.

\bibitem{E} S. Eswara Rao, Partial classification of modules for Lie algebra of diffeomorphisms of d-dimensional torus, J. Math. Phys., 45(2004), 3322-3333.

\bibitem {GSW} M. Green, J. Schwarz, E. Witten, Superstring Theory, Cambridge University Press, 1987.

\bibitem{G} P. Guha, Integrable geodesic flows on the (super)extension of the Bott-Virasoro group, Lett. Math. Phys., 52(2000), 311-328.

\bibitem{GSWX} H. Guo, J. Shen, S. Wang, K. Xu, Beltrami algebra and symmetry of the Beltrami equation on Riemann surfaces, J. Math. Phys., 31(1990), 2543-2547.

\bibitem{IK} K. Iohara, Y. Koga, Representation Theory of the Virasoro algebra, Springer Monographs in Mathematics, Springer-Verlag London Ltd., London, 2011.

\bibitem{KA} V. Kac, A. Raina, Bombay Lectures on Highest weight Representations of Infinte-Dimensional Lie algebras, Advanced Series in Mathematical Physics, Vol. 2,
World Scientific Publishing Co., Inc., Teaneck, NJ, 1987.

\bibitem{KV} V. Kac, J. Van de Leur, On classification of superconformal algebras, Strings, vol. 88, World Scienticfic, Singapore, 1988.

\bibitem{KT} V. Kac, I. Todorov, Superconformal current algebras and their unitary representations, Comm. Math. Phys., 102(1985), 337-347.

\bibitem{LPX} D. Liu, Y. Pei, L. Xia, Classification of quasi-finite irreducible modules over affine-Virasoro algebras, J. Lie Theory, 31(2021), 575-582.

\bibitem{LPXZ} D. Liu, Y. Pei, L. Xia, K. Zhao, Simple smooth modules over the superconformal current algebra, arXiv:2305.16662.

\bibitem{Ma} P. Marcel, Extensions of the Neveu-Schwarz Lie superalgebra, Commun. Math. Phys., 207(1999), 291-306.

\bibitem{Mu} O. Mathieu, Classification of Harish-Chandra modules over the Virasoro Lie algebra, Invent. Math. 107 (1992), no. 2, 225-234.

\bibitem{MOR} P. Marcel, V. Ovsienko, C. Roger, Extension of the Virasoro and Neveu-Schwarz algebras and generalized Sturm-Liouville operators, Lett. Math. Phys., 40(1997), 31-39.

\bibitem{Mo} T. Moons, on the weight spaces of Lie superalgebra modules, J. Algebra, 147(1992), no. 2, 283-323.

\bibitem{PB} Y. Pei, C. Bai, Balinsky-Novikov superalgebras and some infinite-dimensional Lie superalgebras, Journal of Algebra and Its Applications, 6(2012), 1250119.

\bibitem{S} Y. Su, Classification of Harish-Chandra modules over the super-Virasoro algebras. Commun. Algebra 23(1995), no. 10, 3653-3675.

\bibitem{W} E. Witten, Some properties of $O(32)$ superstrings, Phys. Lett., 149B(1984), 351-356.

\bibitem{XL2} Y. Xue, R. L\"{u}, Simple weight modules with finite-dimensional weight spaces over Witt superalgebras, J. Algebra 574(2021), 92-116.


\end{thebibliography}
\end{document}